\documentclass[12pt]{article}

\usepackage{amsmath,amsfonts,amssymb,amscd}
\newtheorem {Lemma}{Lemma}[section]
\newtheorem {Theorem}{Theorem}[section]

\newtheorem {Corollary}{Corollary}[section]
\newenvironment {proof} {\noindent {\bf Proof.}}{\hspace*{\fill}$\Box$\par\vspace{4mm}}

\begin{document}

\title{
A note on the minimum skew rank of  a graph}
\author{
Yanna Wang and Bo Zhou\\
Department of Mathematics, South China Normal
University,\\ Guangzhou 510631, P. R. China\\
e-mail: zhoubo@scnu.edu.cn}
\date{}
\maketitle

\begin{abstract}
The minimum skew rank $mr^{-}(\mathbb{F},G)$ of a graph $G$ over a
field $\mathbb{F}$ is the smallest possible rank among all skew
symmetric matrices over $\mathbb{F}$, whose ($i$,$j$)-entry (for
$i\neq j$) is nonzero whenever $ij$ is an edge in $G$ and is zero
otherwise. We give some new properties of the minimum skew rank of a graph, including
a characterization of the graphs $G$ with cut vertices over the
infinite field $\mathbb{F}$ such that $mr^{-}(\mathbb{F},G)=4$,
determination of the minimum skew rank of $k$-paths over a field $\mathbb{F}$, and an extending of an existing result to show that
$mr^{-}(\mathbb{F},G)=2match(G)=MR^{-}(\mathbb{F},G)$ for a connected graph $G$ with no
even cycles and a field $\mathbb{F}$, where $match(G)$ is the
matching number of $G$, and $MR^{-}(\mathbb{F},G)$ is the largest possible rank among all skew
symmetric matrices over $\mathbb{F}$.
\end{abstract}

%%%%%%%%%%%%%%%%%%%%%%%%%%%%%%%%%%%%%%%%%%%%%%%%%%%%%%%%%%%%%
\section{Introduction}

We consider only simple graphs. Let $G$ be a graph with vertex set
$V_G$ and edge set $E_G$. Let $\mathbb{F}$ be a field. We adopt the
notation and terminology from \cite{AB} and \cite{We}.

An $n\times n$ matrix $A$ over $\mathbb{F}$ is skew-symmetric
(respectively, symmetric) if $A^T=-A$ (respectively, $A^T = A$),
where $A^T$ denotes the transpose of $A$.

For an $n\times n$ symmetric or skew-symmetric  matrix $A$, the
graph of $A$, denoted $G(A)$, is the graph with vertex set
$\{v_{1},v_{2},\dots, v_{n}\}$ and edge set $\{v_{i}v_{j}:
a_{ij}\neq 0, 1\leq i <j \leq n\}$. %The diagonal is ignored in
%determining $G(A)$ if $A$ is  symmetric, while the diagonal must be
%$0$ if $A$ is skew-symmetric.

The classic minimum rank problem involving symmetric matrices has
been studied extensively, see, e.g.,~\cite{FH}.

The minimum skew rank problem involves skew symmetric matrices  and
its study began recently in \cite{AB}. 
%see \cite {LMD,DKT} for more results. 
If the characteristic of $\mathbb{F}$ is $2$, then a
skew-symmetric matrix over $\mathbb{F}$ is also symmetric. %and may
%have nonzero diagonal entries.
Thus it is assumed throughout this
paper that the characteristic of $\mathbb{F}$ is not $2$.

For a field $\mathbb{F}$ and a graph $G$, let
$S^{-}(\mathbb{F},G)=\{A\in \mathbb{F}^{n\times n}: A^{T}=-A,
G(A)=G\}$ be the set of skew-symmetric matrices over $\mathbb{F}$
described by $G$. The minimum skew rank of $G$ over $\mathbb{F}$ is
defined as
\[
mr^{-}(\mathbb{F},G) = \min\{\text{rank}(A): A\in
S^{-}(\mathbb{F},G)\}.
\] The
corresponding maximum skew nullity of $G$ is defined as
\[ M^{-}(\mathbb{F},G)
= \max\{\text{nullity}(A): A\in S^{-}(\mathbb{F},G)\}.
\]
Obviously, $mr^{-}(\mathbb{F},G)+M^{-}(\mathbb{F},G)=|V_G|$.

Let $K_n$ be the complete graph with $n$ vertices, and $K_{n_1,n_2,
\dots, n_t}$ the complete $t$-partite graph with $n_i$ vertices in
the $i$th partite sets for $i=1, 2, \dots, t$.

Note that the rank of a skew-symmetric matrix over $\mathbb{F}$ is
always even. Thus $mr^{-}(\mathbb{F},G)$ is even for any field
$\mathbb{F}$ and any graph $G$. As observed in \cite{AB},
$mr^{-}(\mathbb{F},G)=0$ if and only if $G$ is an empty graph, and
if $\mathbb{F}$ is infinite and $G$ is a connected graph with at
least two vertices, then $mr^{-}(\mathbb{F},G)=2$ if and only if
 $G$ is a complete multipartite graph $K_{n_{1},n_{2},\ldots , n_{t}}$ for some $t\geq 2$,
$n_{i}\geq 1$ for $i=1,\ldots, t$. The authors \cite{AB} posed an
open question (Question 5.2) to characterize the graphs $G$ such
that $mr^{-}(\mathbb{F},G)=4$. We characterize the graphs $G$ with cut vertices over
the infinite field $\mathbb{F}$ such that $mr^{-}(\mathbb{F},G)=4$.

The class of $k$-trees is defined recursively as follows \cite{Ro}:
(i) The complete graph $K_{k+1}$ is a $k$-tree; (ii) A $k$-tree $G$
with $n+1$ vertices ($n\ge k+1$) can be constructed from a $k$-tree
$H$ on $n$ vertices by adding a vertex adjacent to all vertices of a
$k$-clique of $H$.
%
%A $k$-tree is a graph that can be built up from a complete graph $K_{k+1}$ on $k+1$ vertices by
%adding one vertex at a time adjacent to exactly the vertices in an
%existing $K_{k}$.
%A graph is a partial $k$-tree if it is a subgraph of a $k$-tree.
 A $k$-path is a $k$-tree which
is either $K_{k+1}$ or has exactly two vertices of degree $k$. We determine
the minimum skew rank of $k$-paths over a field $\mathbb{F}$. The $k$-th power $G^{k}$ of a graph $G$ is the graph
whose vertex set is $V_G$, two distinct vertices being adjacent in
$G^{k}$ if and only if their distance in $G$ is at most $k$. Let
$P_n=v_1v_2\dots v_n$ be the path on $n$ vertices.
If 
$k\le n-1$, then  $P_{n}^{k}$ is a
$k$-path (see below). As a corollary, we obtain the minimum skew rank
of the $k$-th power of a path  over the real field $\mathbb{R}$, which was already given in \cite{DKT}.

%
%For positive integer $k$, the $k$-th power of a graph $G$, denoted
%by $G^{k}$, is the graph with vertex set $V_G$ such that $uv\in
%E_{G^{k}}$  if and only if distance between $u$ and $v$ in $G$ is at
%most $k$.
% DeAlba et al.~\cite{DKT} determined

The maximum skew rank $MR^{-}(\mathbb{F},G)$ of a graph $G$ over a
field $\mathbb{F}$ is defined as
\[
MR^{-}(\mathbb{F},G) = \max\{\text{rank}(A): A\in
S^{-}(\mathbb{F},G)\}.
\]
Let $match(G)$ be the matching number of  $G$. It was shown
in \cite{AB} that
$mr^{-}(\mathbb{F},G)=2match(G)=MR^{-}(\mathbb{F},G)$ for a tree (a connected graph with no cycles) $G$
and a field $\mathbb{F}$. We extend this by showing that the above conclusion holds also for 
a
connected graph $G$ with no even cycles.

%
%In this paper, we characterize the graphs $G$ with cut vertices over
%the infinite field $\mathbb{F}$ such that $mr^{-}(\mathbb{F},G)=4$, determine
%the minimum skew rank of $k$-paths over a field $\mathbb{F}$, and show that
%$mr^{-}(\mathbb{F},G)=2match(G)=MR^{-}(\mathbb{F},G)$ for a
%connected graph $G$ with no even cycles and a field $\mathbb{F}$.

%%%%%%%%%%%%%%%%%%%%%%%%%%%%%%%%%%%%%%%
\section{Preliminaries}

Let $G$ be a graph. For $v\in V_G$, $G-v$ denotes the graph obtained
from $G$ by deleting vertex $v$ (and all edges incident with $v$).
 For $X\subseteq V_G$, $G[X]$ denotes the
subgraph of $G$ induced by vertices in $X$.

We give some lemmas that we will use in our proof.

\begin{Lemma} \label{lm2.5}
\cite{AB} Let $G$ be a connected graph with at least two vertices
and let $\mathbb{F}$ be an infinite field.
 Then $mr^{-}(\mathbb{F},G)=2$ if and only if
 $G$ is a complete multipartite graph.% = K_{n_{1},n_{2},\ldots , n_{t}}$ for some $t\geq 2$,
%$n_{i}\geq 1$, $i=1,\ldots, t$.
\end{Lemma}

For a field $\mathbb{F}$  and a graph $G$ with $v\in V_G$, let
$r_{v}^{-}(\mathbb{F},
G)=mr^{-}(\mathbb{F},G)-mr^{-}(\mathbb{F},G-v)$.

The union of graphs $G_i$, $i=1,2, \dots,h$, denoted by
$\cup_{i=1}^{h}G_i$, is the graph with vertex set
$\cup_{i=1}^{h}V_{G_i}$ and edge set $\cup_{i=1}^{h}E_{G_i}$.

\begin{Lemma} \cite{AB,De}
 \label{lm2.7}
 Let $G$ be a graph with cut vertex $v$ and  $\mathbb{F}$  a field, where
$G=\cup_{i=1}^{h}G_i$ and $\cap_{i=1}^{h}V_{G_i}=\{v\}$. Then
$mr^{-}(\mathbb{F},G)=\sum_{i=1}^{h}mr^{-}(\mathbb{F},G_{i}-v)+ \min
\{\sum_{i=1}^{h}r_{v}^{-}(\mathbb{F},G_{i}),2\}$.
\end{Lemma}

\begin{Lemma} \label{lm2.1} \cite{AB}
Let $G$ be a graph and let $\mathbb{F}$ be an infinite field. If $G
=G_1\cup G_2$, then $mr^{-}(\mathbb{F},G) \leq
mr^{-}(\mathbb{F},G_1)+mr^{-}(\mathbb{F},G_2)$.
\end{Lemma}

Let $G$ be a graph. A subset $Z\subset V_G$ defines an initial
coloring by coloring all vertices in $Z$ black and all the vertices
outside $Z$ white. The color change rule says: If a black vertex $u$
has exactly one white neighbor $v$, then change the color of $v$ to
black. In this case we write $u\rightarrow v$. The derived set of an
initial coloring $Z$ is the set of vertices colored black until no
more changes are possible. A zero forcing set is a subset $Z\subset
V_G$ such that the derived set of $Z$ is $V_G$. The zero forcing
number of $G$, denoted by $Z(G)$, is the minimum size of a zero
forcing set of $G$.

\begin{Lemma} \label{lm2.4}
\cite{AB} Let $G$ be a graph  and
$\mathbb{F}$  a field. Then  $M^{-}(\mathbb{F},G)\leq Z(G)$.
\end{Lemma}

\begin{Lemma}\label{lm2.8} \cite{AB}
Let $G$ be a graph  and  $\mathbb{F}$  a
field. Then  $MR^{-}(\mathbb{F},G)=2match(G)$.
\end{Lemma}

\begin{Lemma}\label{lm2.9} \cite{AB} Let $G$ be a graph  and  $\mathbb{F}$  a
field. If $H$ is an induced subgraph of $G$,
$mr^{-}(\mathbb{F},H)\leq $ $ mr^{-}(\mathbb{F},G)$.
\end{Lemma}

\begin{Lemma} \label{lm2.6}
\cite{AB} Let $G$ be a graph with a unique perfect matching and
$\mathbb{F}$ a field. Then $mr^{-}(\mathbb{F},G) =|V_G|$.
\end{Lemma}

\section{Results}
%The minimum skew rank of $k$-path and cycle with chords}

we first give a partial result to the question in \cite{AB} to
characterize graphs with $mr^{-}(\mathbb{F},G)=4$. We consider
graphs with cut vertices.

\begin{Theorem} Let $G$
be a graph with cut vertex $v$ and  $\mathbb{F}$ an infinite field.
Then $mr^{-}(\mathbb{F},G)=4$ if and only if one of the following
conditions holds:\\
$(i)$ $G =G_1\cup G_2$ and $V_{G_1}\cap V_{G_2}=\{v\}$, where $G_1$,
$G_2$ are complete multipartite graphs such that $G_{1}-v$,
$G_{2}-v$ are
nonempty, and\\
$(ii)$ $G-v$ consists of a nonempty complete multipartite component
and isolated vertices.
\end{Theorem}

\begin{proof}
Suppose first that  (i) holds. Note that $G_{i}-v$ is still a
complete multipartite graph for $i=1,2$. By Lemma~\ref{lm2.5},
$mr^{-}(\mathbb{F},G_{1})=mr^{-}(\mathbb{F},G_{2})=mr^{-}(\mathbb{F},G_{1}-v)=mr^{-}(\mathbb{F},G_{2}-v)=2$.
Then $r_{v}^{-}(\mathbb{F},G_{1})+r_{v}^{-}(\mathbb{F},G_{2})=0$.
Thus by Lemma~\ref{lm2.7},
$mr^{-}(\mathbb{F},G)=mr^{-}(\mathbb{F},G_{1}-v)+mr^{-}(\mathbb{F},G_{2}-v)+\min\{0,2\}=4$.

Now suppose that  (ii) holds. Let $W$ be the unique complete
multipartite component, and $a$ the number of isolated vertices in
$G-v$. By Lemma~\ref{lm2.5}, $mr^{-}(\mathbb{F},W)=2$. Note that
$r_v^{-}(\mathbb{F},K_{2})=2$. Then by Lemma~\ref{lm2.7},
$mr^{-}(\mathbb{F},G)=mr^{-}(\mathbb{F},W)+a\cdot
mr^{-}(\mathbb{F},K_{1})+2=4$.

Conversely, suppose that  $mr^{-}(\mathbb{F},G)=4$. Let $p$ be the
number of nonempty complete multipartite components, and $q$ the
number of isolated vertices in $G-v$. Let $m$ be the number of the
remaining components.

Note that the minimum skew rank of a graph that is neither a
complete multipartite graph nor an empty graph is larger than $4$.

\noindent {\bf Case 1.} $q=0$. By Lemma~\ref{lm2.7},
$4=mr^{-}(\mathbb{F},G)\geq 2p+4m$. If $m=1$, then $p=0$, a
contradiction to the fact that $v$ is a cut vertex of $G$. Thus
$m=0$, implying that $p=2$. Let $W_{1}$, $W_{2}$ be the vertex sets
of the two complete multipartite  components of $G-v$ and let
$G_{1}$, $G_{2}$ be the subgraph induced by $\{v\}\cup W_{1}$,
$\{v\}\cup W_{2}$. By Lemma~\ref{lm2.5},
$mr^{-}(\mathbb{F},G_{1}-v)=mr^{-}(\mathbb{F},G_{2}-v)=2$. By
Lemma~\ref{lm2.7},
$4=mr^{-}(\mathbb{F},G)=mr^{-}(\mathbb{F},G_{1}-v)+mr^{-}(\mathbb{F},G_{2}-v)+\min\{r_{v}^{-}(\mathbb{F},G_{1})+r_{v}^{-}(\mathbb{F},G_{2}),2\}=2+2+\min\{r_{v}^{-}(\mathbb{F},G_{1})+r_{v}^{-}(\mathbb{F},G_{2}),2\}$.
Then $r_{v}^{-}(\mathbb{F},G_{1})=r_{v}^{-}(\mathbb{F},G_{2})=0$.
Thus $mr^{-}(\mathbb{F},G_{1})=mr^{-}(\mathbb{F},G_{2})=2$.  By
Lemma~\ref{lm2.5}, $G_1$ and $G_2$ are complete multipartite graphs,
and then (i) follows.

\noindent {\bf Case 2.} $q\neq0$. Note that
$r_v^{-}(\mathbb{F},K_{2})=2$. By Lemma~\ref{lm2.7},
$4=mr^{-}(\mathbb{F},G)\geq 2p+4m+2$. Then $m=0$ and $p=1$, and thus
(ii) follows.
\end{proof}

Now we consider the minimum skew rank of $k$-paths.  Note that a
$k$-path with at least $k+2$ vertices has at least two vertices of
degree $k$ and any two vertices of degree $k$ are not adjacent. The
following lemma follows directly from the definition of $k$-path.

\begin{Lemma} \label{base} Let $G$ be a $k$-path with at least $k+2$
vertices, and $v$ a vertex of $G$ with degree $k$. Then $G-v$ is
also a $k$-path.
\end{Lemma}

%\begin{proof} If $|V(G)|=k+2$, then it is obvious. Suppose that $|V(G)|\ge k+3$.
%Obviously, $G$ is a $k$-tree. Thus $G-v$ is a $k$-tree. Let $u$ be
%any neighbor of $v$ in $G$. Suppose that $d_G(u)\ge k+2$. Then
%$d_{G-v}(u)\ge k+1$. Note that $G-v$ has at least two vertices of
%degree $k$. Thus $G$ has at least three vertices of degree $k$, a
%contradiction. It follows that $v$ has at least one neighbor in $G$
%of degree $k+1$.  If $v$ has two neighbors of degree $k+1$ in $G$,
%then they become adjacent vertices of degree $k$ in $G-v$, which is
%impossible because $G-v$ is a $k$-tree. Thus $v$ has exactly one
%neighbor of degree $k+1$ in $G$. Then $G-v$ has exactly two vertices
%of degree $k$, and thus $G-v$ is a $k$-path.
%\end{proof}

Let $G$ be a $k$-path with $n\ge k+2$ vertices. By Lemma~\ref{base},
 the vertices of $G$ may be labeled as follows: Choose a vertex
of degree $k$, labeled as $v_n$, and label its unique neighbor of
degree $k+1$ in $G$ with $v_{n-1}$. Then $v_{n-1}$ is a vertex of
degree $k$ in the $k$-path  $G-v_n$. Repeating the process above, we
may label $n-{k+1}$ vertices of $G$ as $v_n$, $v_{n-1}, \dots,
v_{k+2}$. Obviously, $G-v_n-v_{n-1}-\cdots -v_{k+2}=K_{k+1}$ and it
contains a vertex of degree $k$ in $G$, which is labeled as $v_1$,
and the remaining vertices are labeled as $v_2$, $v_3, \dots,
v_{k+1}$ such that $v_2$ is the unique neighbor of $v_1$ with degree
$k+1$ in $G$. Note that in our labelling, $v_i$ is not adjacent to
$v_{j+1}, v_{j+2},\dots, v_n$ if $v_i$ is not adjacent to $v_j$ for
$j\ge \max\{i+1,k+2\}$.  Recall that a $k$-tree is a chordal graph. The above labeling is
actually the ``perfect elimination" labeling inherent to chordal graphs \cite{Sh}.% We
%use this labelling in the proof of the following theorem.

\begin{Theorem} \label{th1}
Let $G$ be a $k$-$path$ on $n$ vertices and  $\mathbb{F}$  an
infinite field. Then
\begin{equation*}
mr^{-}(\mathbb{F},G)=\begin{cases}n-k &\text{if $n-k$ is even},\\
                       n-k+1 &\text{if $n-k$ is odd}.
\end{cases}
\end{equation*}
\end{Theorem}

\begin{proof}
Let $Z=\{v_{1},v_{2},\ldots , v_{k}\}$. Color all vertices in $Z$
black and all the vertices outside $Z$ white. We will show that $Z$
is a zero forcing set of $G$. Since all neighbors of $v_{1}$
different from $v_{k+1}$ are black, we have $v_{1}\rightarrow
v_{k+1}$. Note that $v_{2}$ is adjacent to $v_{k+2}$ but not
adjacent to $v_{k+3},v_{k+4},\ldots, v_{n}$.
 Since all neighbors of $v_{2}$
different from $v_{k+2}$ are black, we have $v_{2}\rightarrow
v_{k+2}$. Let $G_1=G[\{v_{1},v_{2},\ldots , v_{k+3}\}]$ and
$G_2=G[\{v_{1},v_{2},\ldots , v_{k+4}\}]$. If each neighbor of
$v_{k+3}$ in $G_1$ is adjacent to $v_{k+4}$ in $G$, then $v_{k+4}$
is of degree $k+1$ in $G_2$, a contradiction. Thus there is a
neighbor, say $w$, of $v_{k+3}$ in $G_1$ such that $wv_{k+4}\not\in
E_G$, and then $wv_i\not\in E_G$ for $i\ge k+5$, implying that
 $w\rightarrow v_{k+3}$.
 Repeating the process above, we may finally color all vertices of $G$ black. Thus $Z$ is
a zero forcing set of $G$.  By Lemma~\ref{lm2.4},
$M^{-}(\mathbb{F},G)\leq Z(G)\leq k$, and then
$mr^{-}(\mathbb{F},G)=n-M^{-}(\mathbb{F},G)\geq n-k$. Note that the
rank of a skew-symmetric matrix is even. It follows that
\begin{eqnarray*}
mr^{-}(\mathbb{F},G)\geq\begin{cases}n-k &\text{if $n-k$ is even},\\
                         n-k+1 &\text{if $n-k$ is odd}.
\end{cases}
\end{eqnarray*}

To prove the result, we need only to show
\begin{equation}\label{ee}
mr^{-}(\mathbb{F},G)\leq\begin{cases}n-k &\text{if $n-k$ is even},\\
                         n-k+1 &\text{if $n-k$ is odd}.
\end{cases}
\end{equation}
We prove this by induction on $n$. If $n=k+1$, then $G=K_{k+1}$,
which is a complete multipartite graph, and thus by
Lemma~\ref{lm2.5}, $mr^{-}(\mathbb{F},G)=2=n-k+1$. If $n=k+2$, then
$G=K_{k+2}-e$ is also a complete multipartite graph, where $e\in
E_{K_{k+2}}$, and thus by Lemma  \ref{lm2.5},
$mr^{-}(\mathbb{F},G)=2=n-k$. Thus (\ref{ee}) is true for $n=k+1,
k+2$. Suppose that $n\geq k+3$ and for a $k$-path $H$ on $m$
vertices with $k+1\leq m\leq n-1$, we have
\begin{equation*}
mr^{-}(\mathbb{F},H)\leq\begin{cases}m-k &\text{if $m-k$ is even},\\
                         m-k+1 &\text{if $m-k$ is odd}.
\end{cases}
\end{equation*}
Let $G$ be a $k$-path on $n$ vertices. Let
$$G_{1}=G[\{v_{1},v_{2},\ldots ,v_{k+2}\}] \mbox{ and }
G_{2}=G[\{v_{3},v_{4},\ldots ,v_{n}\}].$$ Then $G_{1}$ is a $k$-path
on $k+2$ vertices, and $G_{2}$ is a $k$-path on $n-2$ vertices.
Obviously, $mr^{-}(\mathbb{F},G_{1})=2$, and by the induction
hypothesis,
\begin{equation*}
mr^{-}(\mathbb{F},G_{2})\leq\begin{cases}n-k-2 &\text{if $n-k-2$ is even},\\
                         n-k-1 &\text{if $n-k-2$ is odd},
\end{cases}
\end{equation*}
i.e.,
\begin{equation*}
mr^{-}(\mathbb{F},G_{2})\leq\begin{cases}n-k-2 &\text{if $n-k$ is even},\\
                         n-k-1 &\text{if $n-k$ is odd}.
\end{cases}
\end{equation*}
Note that  $G=G_{1}\cup  G_{2}$. By Lemma~\ref{lm2.1},
\begin{eqnarray*}
mr^{-}(\mathbb{F},G)&\le& mr^{-}(\mathbb{F},G_{1})+ mr^{-}(\mathbb{F},G_{2})\\
         &\leq& 2+\begin{cases}n-k-2 &\text{if $n-k$ is even}\\
                         n-k-1 &\text{if $n-k$ is odd}
                         \end{cases}\\
         &=&\begin{cases}n-k &\text{if $n-k$ is even},\\
                         n-k+1 &\text{if $n-k$ is odd}.
                         \end{cases}
\end{eqnarray*}
This proves (\ref{ee}).
\end{proof}

%Let $\mathbb{R}$ be the real field.

%Recall that the $k$-th power $G^{k}$ of a graph $G$ is the graph
%whose vertex set is $V_G$, two distinct vertices being adjacent in
%$G^{k}$ if and only if their distance in $G$ is at most $k$. Let
%$P_n=v_1v_2\dots v_n$ be the path on $n$ vertices.
Obviously, $P_n^k$ is a complete graph if $k\ge n$. Suppose that
$k\le n-1$. Obviously, $P_n^k[\{v_1,v_2, \dots, v_{k+1}\}]=K_{k+1}$,
and if $k\le n-2$, then for $j=2,3,\dots,n-k$,
$P_n^k[\{v_{j},v_{j+1}, \dots, v_{k+j-1}\}]=K_{k}$, and $v_{k+j}$ is
adjacent to $v_{j},v_{j+1}, \dots, v_{k+j-1}$. Thus $P_{n}^{k}$ is a
$k$-path. Now by Lemma~\ref{lm2.5} and Theorem
\ref{th1} we have the following result, which was proved in \cite{DKT} when $\mathbb{F}$ is the real field $\mathbb{R}$.

\begin{Corollary} Let $\mathbb{F}$  be an
infinite field. Then
 % and let $P_{n}^{k}$ be the $k$-th power of $P_{n}$.
\begin{eqnarray*}
mr^{-}(\mathbb{F},P_{n}^{k})=\begin{cases}n-k &\text{if $1\le k\le n-1$ and $n-k$ is even},\\
                         n-k+1 &\text{if $1\le k\le n-1$ and $n-k$ is
                         odd},\\
                         2 &\text{if $k\ge n$}.

\end{cases}
\end{eqnarray*}
\end{Corollary}

%\begin{proof} Let $P_n=v_1v_2\dots v_n$.
%Obviously, $P_n^k$ is a complete graph if $k\ge n$. Suppose that
%$k\le n-1$. Obviously, $P_n^k[\{v_1,v_2, \dots, v_{k+1}\}]=K_{k+1}$,
%and if $k\le n-2$, then for $j=2,3,\dots,n-k$,
%$P_n^k[\{v_{j},v_{j+1}, \dots, v_{k+j-1}\}]=K_{k}$, and $v_{k+j}$ is
%adjacent to $v_{j},v_{j+1}, \dots, v_{k+j-1}$. Thus $P_{n}^{k}$ is a
%$k$-path. Then the result follows from Lemma~\ref{lm2.5} and Theorem
%\ref{th1}.
%\end{proof}

Finally, we gave an observation.

\begin{Theorem} \label{add}
Let  $G$ be a connected  graph with no even cycles and $\mathbb{F}$
a field. Then
 $mr^{-}(\mathbb{F},G)=2match(G)=MR^{-}(\mathbb{F},G)$.
\end{Theorem}

\begin{proof}
By Lemma \ref{lm2.8}, $mr^{-}(\mathbb{F},G)\leq
MR^{-}(\mathbb{F},G)=2match(G)$. Let $M$ be a maximum matching of
$G$ and $\{v_{1},\cdots,v_{k}\}$ the vertices in $M$. Then $M$ is a
perfect matching of $H=G[\{v_{1},\cdots,v_{k}\}]$. This perfect
matching is unique. Otherwise, the graph induced by the vertices of
the symmetric difference of two (different) perfect matchings of $H$
consists of even cycles, which is impossible because $G$ contains no
even cycles. By Lemmas \ref{lm2.9} and  \ref{lm2.6},
$mr^{-}(\mathbb{F},G)\geq mr^{-}(\mathbb{F},H)=2match(G)$. The
result follows.
\end{proof}

Note that a tree has no (even) cycles. By previous theorem we have the following result.

\begin{Corollary} \cite{AB}
Let $G$ be a tree and  $\mathbb{F}$  a field. Then
$mr^{-}(\mathbb{F},G)=2match(G)$ $=MR^{-}(\mathbb{F},G)$.
\end{Corollary}

 Let $G$ be a connected unicyclic graph with a unique cycle $C$. If $C$ is odd, then by Theorem \ref{add},
 $mr^{-}(\mathbb{F},G)=2match(G)$. Recall that it was shown in \cite{LMD} that if $C$ is odd, then 
 $mr^{-}(\mathbb{R},G)=2match(G)$, and if $C$ is even, then
 $mr^{-}(\mathbb{R},G)=2match(G)$ or $2match(G)-2$.

\bigskip
{\bf Acknowledgment.} This work was supported by the National
Natural Science Foundation of China (No.~11071089)and the Guangdong
Provincial Natural Science Foundation of China (no.~S2011010005539).

%%%%%%%%%%%%%%%%%%%%%%%%%%%%%%%%%%%%%%%%%%%%%%%%%%%%%%%%%%%%%

\end{document}